\documentclass[11pt]{amsart}  
\usepackage[cccs]{ABT}

\newtheorem*{aside}{Aside}

\newcommand{\interior}[1]{\overset{\circ}{#1}}

\newcommand{\Uf}[2]{U(#2;#1)}
\newcommand{\Ufn}[3]{U(#3;#1\downarrow#2)}
\newcommand{\Ufrestrn}[3]{U(#3;#1\restr\,#2)}
\newcommand{\UXf}[3]{U^{#1}(#3;#2)}
\newcommand{\minsub}[2]{#1^{\min}_{#2}}

\title{Products of CW complexes} 
\author{Andrew D. Brooke-Taylor}
\address{School of Mathematics\\
University of Leeds\\
Leeds LS2 9JT\\
United Kingdom}
\date{\today}

\begin{document}
\maketitle
\begin{abstract}
We provide for the first time a complete characterisation  of when the
product as topological spaces of two CW complexes is a CW complex, 
valid irrespective of set theory.
The previous state-of-the-art
was that, under set-theoretic assumptions like the Continuum Hypothesis, 
it is necessary and sufficient that a criterion of Whitehead or a
criterion of  Milnor holds.
The general case without making such assumptions is not quite so
simple, but we show that it
still essentially reduces to counting the number of cells.
\end{abstract}

\section{Introduction} 
CW complexes are considered to be well-behaved spaces that provide a convenient setting for the development of algebraic topology. 
Indeed, their dimension-by-dimension construction from Euclidean cells rules out many point-set-theoretic pathologies which are inessential to fundamental questions of homotopy theory, such as computation of the homotopy groups of spheres. 
Furthermore, since every topological space is weakly equivalent to a CW complex, no homotopy-theoretic information is lost by restricting attention to CW complexes.

The category of CW complexes does have drawbacks, however.
A well-known quirk is that the inclusion of this category into the category
of topological spaces does not preserve products:
the product as topological spaces of two CW complexes $X$ and $Y$
need not be a CW complex.  
Given two CW complexes $X$ and $Y$, 
there is a natural cell structure on their Cartesian product $X\times Y$
which satisfies closure-finiteness (the ``C'' of ``CW''), 
but the compactly generated
topology required for 
a CW complex is in general finer than
the product topology.  
In the 1949 paper introducing CW complexes \cite[Section~5]{Whi:CHI}, 
Whitehead showed that if
one of the two CW complexes is locally finite 
(see Definition~\ref{locltka} below),
the product is a CW complex, adding in a footnote 
``I do not know if this restriction [\,\ldots] is necessary''.
Soon afterwards Dowker \cite{Dow:TMC} 
demonstrated that \emph{some} assumption
on $X$ and $Y$ was necessary for the product to be a CW complex:
Dowker exhibited a CW complex $X$ with countably many cells and a CW complex $Y$
with continuum ($2^{\aleph_0}$) many cells 
such that $X\times Y$ is not a CW complex.
On the other hand, Milnor~\cite{Mil:CUBI} showed that Whitehead's
assumption itself was not necessary, as the product of any two CW complexes each
with countably many cells is a CW complex.

Subsequent advances have relied on making assumptions not just about the
CW complexes $X$ and $Y$, but about the whole domain of discourse: 
the model of set theory in which
$X$ and $Y$ lie.
In models of set theory in which 
the Continuum Hypothesis (the assumption that $2^{\aleph_0}=\aleph_1$) holds, 
Liu Ying-Ming \cite{LYM:NSCPCW}
showed that the product of
CW complexes $X$ and $Y$ will be a CW complex if and only if $X$ or $Y$ is 
locally finite or both are locally countable --- that is, if and only if either
Whitehead or Milnor's criterion holds.
Building on work of Gruenhage \cite{Gru:kSPCIMS}, 
Tanaka \cite{Tan:PCW} showed that 
the validity of this characterisation is equivalent to a 
weaker set-theoretic assumption
there denoted ``not $\textrm{BF}(\omega_2)$''; in current set-theoretic
notation, this is $\frb=\aleph_1$.

The present paper provides a complete answer to the question of when 
the product of a pair of CW complexes is again a CW complex.
No extra set-theoretic assumptions are required for our characterisation,
and if one does make the set-theoretic assumptions employed by Liu or Tanaka,
our characterisation reduces to theirs.
\begin{thm}\label{prodCWcharthm}
Let $X$ and $Y$ be CW complexes.  Then $X\times Y$ is a CW complex if and only if
one of the following holds:
\begin{enumerate}
\item Either $X$ or $Y$ is locally finite.
\item\label{lltb}
Either $X$ or $Y$ has countably many cells in each connected component, and the
other has fewer than $\frb$ many cells in each connected component.
\end{enumerate}
\end{thm}
The cardinal $\frb$ in alternative (\ref{lltb}) is the so-called
\emph{bounding number}: 
the least cardinality of a set of functions from $\N$ to $\N$ which is 
unbounded with respect to eventual domination (Definition~\ref{bdefn}). 
This relatively concrete definition for a cardinal has been extensively studied
by set theorists, and $\frb$ is a standard cardinal in the pantheon of 
\emph{cardinal characteristics of the continuum} (for more on which see,
for example, Blass \cite{Bla:CCCC}).
The value of $\frb$ in terms of Cantor's $\aleph$ hierarchy can vary 
between different models of set theory, much as the cardinality 
of the reals $2^{\aleph_0}$ can, 
so a corollary of our result is that no characterisation purely in terms of
the $\aleph$ hierarchy can be valid in all models of set theory.
We will introduce $\frb$ in more detail in Section~\ref{prelimsb} below;
our proof and presentation
require no background in set theory of the reader.
We also review the basics of CW complexes in Section~\ref{prelimsCW} for 
readers unfamiliar with them; further details, including a subsection on
products where Dowker's example is presented, may be found in the Appendix
of the standard text of Hatcher \cite{Hat:AT}.

The results of Tanaka~\cite{Tan:PCW} rely on reducing to the case where the 
CW complexes in question are ``stars'' consisting of a 
central vertex and edges emanating from it.  In this one-dimensional setting,
the ``finitely many errors'' aspect of eventual domination is unproblematic, 
thanks to local compactness away from the ``bad'' central point.
For our full result no such reduction is 
possible, and ``finitely many errors'' 
can multiply to infinitely many when cells
of higher dimension are attached.
It turns out that eventual domination 
is nevertheless the right notion in this case:
whilst a na\"{\i}ve induction does not suffice to prove the 
theorem,
the proof of our Lemma~\ref{fnforal} 
shows that with sufficient care, an induction 
incorporating promises about the growth rate at later stages of the 
induction can be made to work.  Interestingly, such promises are also used in
\emph{Hechler forcing}, 
the typical set-theoretic technique used for building models of
set theory with a chosen value of $\frb$ (although no knowledge of Hechler 
forcing is needed for our arguments).

\section{Preliminaries}\label{prelims}

\subsection{The cardinal $\frb$}\label{prelimsb}

Recall that by definition $\aleph_1$ is the least uncountable cardinal, and
$2^{\aleph_0}$ is the cardinality of the reals; Cantor showed that 
$2^{\aleph_0}\geq\aleph_1$.  There are a number of well-studied
definitions for cardinals
that lie between these two bounds --- so-called \emph{cardinal characteristics of
the continuum}.  For an introduction to them see, for example, \cite{Bla:CCCC}.
Of relevance to us is the bounding number $\frb$.  
\begin{defn}\label{eventualdom}
Given two functions
$f$ and $g$ from $\N$ to $\N$, 
say that $f$ is \emph{eventually dominated} by $g$, written
$f\leq^*g$, if for all but finitely many $n$ in $\N$, $f(n)\leq g(n)$.
\end{defn}
We also write $f\leq g$ to mean that $f(n)\leq g(n)$ for all $n\in\N$.
\begin{defn}\label{bdefn}
The \emph{bounding number} $\frb$ is the least cardinality of
a set of functions $\N\to\N$ that is unbounded with respect to eventual 
domination, that is,
\[
\frb=\min\{|\calF|\st \calF\subseteq \N^{\N}\land 
\forall g\in \N^\N \exists f\in\calF\ \lnot(f \leq^*g)\}.
\]
\end{defn}
It is a simple exercise to show that $\frb$ is uncountable,
and obviously $\frb$ is at most the cardinality of $\N^\N$, 
so $\aleph_1\leq\frb\leq 2^{\aleph_0}$.  
If the Continuum Hypothesis holds then of course
$\aleph_1=\frb=2^{\aleph_0}$, but there are also models of set theory
(that is, universes of sets satisfying the standard ZFC axioms)
in which $\aleph_1=\frb<2^{\aleph_0}$, models in which
$\aleph_1<\frb=2^{\aleph_0}$, and models in which $\aleph_1<\frb<2^{\aleph_0}$.

Recall that a cardinal 
$\ka$ is said to be \emph{singular} if it can be expressed as a
``small'' union of ``small'' sets: 
\[
\ka=\bigcup_{\al<\ga}I_\al,
\]
with $\ga<\ka$ and $|I_\al|<\ka$ for each $\al<\ga$.  A cardinal is said to be
\emph{regular} otherwise.
\begin{lem}[Folklore]
The bounding number $\frb$ is regular.
\end{lem}
\begin{proof}
Let $X$ be a set of functions from $\N$ to $\N$ of cardinality $\frb$ 
which is unbounded with respect to eventual domination; 
enumerate $X$ as $X=\{f_\ze\st\ze\in\frb\}$.  
Suppose for the sake of contradiction that $\frb$ can be decomposed as
$\frb=\bigcup_{\al<\ga}I_\al$ with
$\ga<\frb$ and $|I_\al|<\frb$ for every $\al<\ga$.
Then for each $\al$ there must be some $g_\al\from\N\to\N$ eventually
dominating each member of $\{f_\ze\st \ze\in I_\al\}$.
But then $\{g_\al\st\al<\ga\}$ would be an unbounded set of functions of
cardinality $\ga<\frb$, contradicting the minimality of $\frb$.
\end{proof}

\subsection{CW complexes}\label{prelimsCW}

We review the basics of CW complexes here; for more detail see, for example,
the Appendix of \cite{Hat:AT}.  We also present a result
(Proposition~\ref{lltkacomponents}) that allows
the formulation of our main theorem 
in terms of numbers of cells in connected components, as it is
given in Theorem~\ref{prodCWcharthm}. 

The $n$-disc $D^n$ is the closed unit ball in $\R^n$.
Its interior is here denoted $\interior{D^n}$ and
its boundary is the $(n-1)$-sphere $S^{n-1}$.
We denote the image of a function $\phi$ pointwise on a set $X$ by $\phi[X]$.

\begin{defn}\label{CWdefn}
A Hausdorff space $X$ is a \emph{CW complex} if there exist 
continuous functions $\phi^n_\al:D^n\to X$ (\emph{characteristic maps}) 
for $\al$ in an arbitrary index set and $n\in\N$
a function of $\alpha$, such that 
the following conditions hold.
\begin{enumerate}
\item\label{CWunion} The restriction $\phi^n_\al\restr\interior{D^n}$ 
of $\phi^n_\al$ to the interior
of $D^n$ is a homeomorphism to its image, 
and $X$ is the disjoint union as $\al$ varies of these homeomorphic images
$\phi^n_\al[\interior{D^n}]$. 
We denote $\phi^n_\al[\interior{D^n}]$ by $e^n_\al$ and refer to it as an
\emph{$n$-dimensional cell}.
\item\label{CWclfin} (\emph{Closure-finiteness})
For each $\phi^n_\al$, the image $\phi^n_\al[S^{n-1}]$
of the boundary of $D^n$ is contained in finitely many cells all of dimension less
than $n$.
\item\label{W} 
The topology on $X$ is the \emph{weak topology}: a set is closed if and only
if its intersection with each closed cell $\phi^n_\al[D^n]$
is closed.\footnote{That is, the topology which is finest
possible while
having all of the functions $\phi^n_\al$ continuous.}
\end{enumerate}
\end{defn}

A compact subset of a CW complex is contained in finitely many cells $e^n_\al$,
and each closed cell $\bar{e}^n_\al=\phi^n_\al[D^n]$ is compact, so (\ref{W}) is equivalent to 
the statement that the topology is \emph{compactly generated} --- a set is
closed if and only if its intersection with each compact set is closed.
We can also restrict to compact sets of the form of a convergent sequence with 
its limit point: a space is said to be \emph{sequential} if each subset $Y$ is
closed precisely when $Y$ contains the limit point of every convergent countable
sequence contained in $Y$.  Since each $\phi^n_\alpha[D^n]$ is sequential, 
a Hausdorff space satisfying
criteria (\ref{CWunion}) and (\ref{CWclfin}) of Definition~\ref{CWdefn}
is a CW complex if and only if it is sequential.

Compactly generated spaces are sometimes referred to as $k$-spaces
(for example in \cite{Gru:kSPCIMS} and \cite{Tan:PCW}), and 
compactly generated Hausdorff spaces as Kelley spaces
(for example in \cite{GaZ:CFHT}).  We have elected to stick with the more
descriptive ``compactly generated'' terminology used, for example,
in \cite{Hat:AT}.

A subcomplex $A$ of a CW complex $X$ is a subspace which is a union of cells of
$X$, such that if $e^n_\al$ is contained in $A$ then its closure 
$\phi^n_\al[D^n]$ is contained in $A$.  Clearly such an $A$ forms a CW complex
with those characteristics maps $\phi^n_\al$ from the CW structure on $X$ 
corresponding to the cells of $A$.
An important example of a subcomplex is the \emph{$n$-skeleton}
$X^n$ of $X$ for any given $n\in\N$: $X^n$ is
the union of all cells $e^m_\al$ of $X$ of
dimension $m\leq n$. 
Every subcomplex of $X$ is closed in $X$.

Each cell $e_\al$ of a CW complex $X$, and hence each point $x$ in $X$, 
is contained in a finite subcomplex. 
Indeed there is clearly a least such subcomplex, constructible by starting
with $e_\al$ and working down through the dimensions adding only those cells
necessary to contain the boundaries of those that came before.
\begin{defn}\label{Xminal}
For $X$ a CW complex and $e_{X,\al}$ a cell of $X$, we denote by
$\minsub{X}{\al}$ the minimal (with respect to inclusion) subcomplex of $X$
containing $e_\al$
\end{defn}
There is more variability, however, if we want $x$ to be in the 
\emph{interior} of the subcomplex.
\begin{defn}\label{locltka}
Let $\ka$ be a cardinal.  A CW complex $X$ is said to be
\emph{locally less than $\ka$} if for all $x$ in $X$ there is a subcomplex
$A$ of $X$ with fewer than $\ka$ many cells such that $x$ is in the interior of
$A$.
We write \emph{locally finite} for locally less than $\aleph_0$, and 
\emph{locally countable} for locally less than $\aleph_1$.
\end{defn}

\begin{prop}\label{lltkacomponents}
Let $\ka$ be an uncountable regular cardinal.  Then a CW complex $X$ is
locally less than $\ka$ if and only if each connected component of $X$ contains
fewer than $\ka$ many cells.
\end{prop}

\begin{proof}
Clearly if 
each connected component of a CW
complex $X$ has fewer than $\ka$ many cells, then $X$ is locally less than
$\ka$. 
For the converse, let $\ka$ be an uncountable regular cardinal,
let $X$ be a locally less than $\ka$ CW complex, and consider any point $x$ in
$X$; we shall show that the connected component of $X$ containing $x$ 
contains fewer than $\ka$ many cells.  
The argument is by a recursive construction of the component.

For the base case,
take a connected subcomplex $A_1$ of $X$ with fewer than
$\ka$ many cells which contains $x$ in its interior.  For notational
convenience let $A_{0}$ be the empty space.
Now suppose we have defined $A_i$ a subcomplex of $X$ with 
fewer than $\ka$ many cells,
with $x$ in its interior,
and moreover with every point of $A_{i-1}$ in its interior.
Consider a cell $e$ of $A_i$.
Since $X$ is locally less than $\ka$,
we may choose for every point $y$ of $\bar{e}$
a connected subcomplex $A_y$ of $X$ with fewer than $\ka$ many cells 
and
an open set $U_y$ of $X$ 
such that
$y\in U_y\subseteq A_y$.
In particular, $U_y\cap\bar{e}$ is open in $\bar{e}$, and if 
$z\in U_y\cap\bar{e}$ then certainly $z$ is in the interior of $A_y$.  
Since $\bar{e}$ is compact, a finite set $S_e$ of points $y$ 
suffice for the sets $U_y\cap\bar{e}$ to cover $\bar{e}$.
Take
\[
A_{i+1}=\bigcup_{\substack{e\text{ a cell}\\ \text{of }A_i}}\ \bigcup_{y\in S_e}A_y.
\]
Each $A_y$ has fewer than $\ka$ many cells, and the union is over fewer than
$\ka$ many indices, so by regularity of $\ka$,
$A_{i+1}$ has fewer than $\ka$ many cells.
Each $A_y$ in the union is connected to $A_i$, so $A_{i+1}$ is connected,
and by construction $A_i$ is contained in the interior of $A_{i+1}$, so
the inductive step is complete.

Finally, let $A=\bigcup_{i\in\N}A_i$.  Since $\ka$ is regular uncountable
and each $A_i$ has fewer than $\ka$ many cells, $A$ has fewer than $\ka$ 
many cells.  Since $A$ is an increasing union of connected spaces it is 
connected; it is open by construction and closed as a subcomplex of $X$,
so it is a connected component of $X$.
\end{proof}

We note that Tanaka~\cite[Lemma~4]{Tan:PCW},
citing the paper \cite{LYM:NSCPCW} (in Chinese) of Liu,
misstates this Proposition for local countability,
with an erroneous point- and cell-wise version.  As an example, take the closed
unit interval $I$ built as usual as an edge $e$ connecting two vertices. 
Attach a 2-sphere to a point $x$ in $e$
(that is, attach a 2-disc using the constant characteristic map taking every
point of its boundary circle to $x$), and attach
a $3$-sphere to every other point of the $2$-sphere. 
Then every open neighbourhood in $e$ of $x$ intersects the boundary of
only one other cell, but this CW complex is not locally countable at $x$.
Fortunately the global version, Proposition~\ref{lltkacomponents} with
$\ka=\aleph_1$, suffices for the proofs of Tanaka's results.

We will also require the following basic result, a proof 
of which may be found in \cite[Proposition~A.3]{Hat:AT}.

\begin{lem}[{Whitehead~\cite[\S5(G)]{Whi:CHI}}]
CW complexes are normal.
\end{lem}

\section{The characterisation}

One direction of Theorem~\ref{prodCWcharthm} is immeditate from 
the following theorem of Tanaka \cite{Tan:PCW}.
\begin{thm}[Tanaka]\label{TanakaThm}
The following are equivalent.
\begin{enumerate}
\item\label{kagtb} $\ka\geq\frb$
\item\label{prodCW} If $X\times Y$ is a CW complex, then either
\begin{enumerate}
\item\label{onelf} $X$ or $Y$ is locally finite, or
\item\label{lclltka} $X$ or $Y$ is locally countable and the other is locally less than $\ka$.
\end{enumerate}
\end{enumerate}
\end{thm}
In light of Proposition~\ref{lltkacomponents},
in order to prove Theorem~\ref{prodCWcharthm} 
it remains to show that if $\ka=\frb$, the disjunction
((\ref{onelf}) or (\ref{lclltka})) of statements from Theorem \ref{TanakaThm}
implies that $X\times Y$ is a CW complex.
The fact that if $X$ or $Y$ is locally finite then $X\times Y$ is a CW complex
was shown by Whitehead \cite[\S5]{Whi:CHI}, and
follows from the fact that 
the Cartesian product of a compactly generated space with a 
locally compact space is compactly generated
(see, for example, \cite[Proposition~A.15]{Hat:AT})
--- locally finite CW complexes are clearly locally compact.

It is therefore incumbent on us to show that if $X$ is locally countable
and $Y$ is locally less than $\frb$ then the product of $X$ and $Y$ is a CW 
complex.  We may consider individual connected components, so for the rest of
this section, we shall assume that $X$ is a CW complex with countably many
cells, and $Y$ is a CW complex with fewer than $\frb$ many cells.

We now give some definitions that will be useful for the proof.
For the sake of expositional clarity various dependencies will be omitted from
the notation.

We follow the standard notation from set theory, that when a natural number $n$
is used in place of a set of natural numbers, it denotes the 
$n$-element set $\{0,\ldots, n-1\}$.
For a function $s:I\to K$, the function that extends $s$ by taking value
$q$ on some $\al\notin I$ is denoted by $s\cup\{(\al,q)\}$.

We start by defining 
a descending sequence of neighbourhoods $B_n(x)$ open in a cell $e$ 
that form a neighbourhood base in $e$ of a point $x$.

\begin{defn}\label{Bphirx}
Suppose $x$ is a point in a CW complex $X$, with $x$ lying in an open cell $e$ of 
dimension $d$ with characteristic map $\phi$, and suppose $n$ is a natural number.
Let $z$ be $\phi^{-1}(x)$, and let $r\in\R$ be the minimum of $1/(n+1)$ and 
half the distance from $z$ to the boundary of $D^d$.  Then we define
$B_n(x)$ to be the image under $\phi$ of the open ball of radius $r$ about
$z$ in $D^d$.
\end{defn}
The set $B_n(x)$ need not be open as a subset of $X$ ---
to build an open neighbourhood in $X$ we must also
consider higher-dimensional cells whose boundaries intersect $B_n(x)$.
For these cells we use the following ``collar neighbourhoods''.

\begin{defn}\label{CnU}
Let $X$ be a CW complex, $d$ a natural number, 
and $U\subseteq X^d$ a subset of $X^{d}$ which is open
in $X^d$.
Let $e$ be a $(d+1)$-dimensional cell of $X$ with characteristic map
$\phi$, and let $n$ be a natural number.  We define the open subset 
$C^e_n(U)$ of $\bar e$ by
\[
C^e_n(U)=\phi\left[\left\{t\cdot\vec{z}\st t\in(\tfrac{n}{n+1},1]\land
\vec{z}\in \phi^{-1}(U)\subseteq S^d\right\}\right],
\]
where the $\cdot$ denotes scalar multiplication in the vector space $\R^{d+1}$.
\end{defn}

Note that 
if $\phi^{-1}(U)$ is empty then
$C^e_n(U)$ will also be empty, and that $C^e_n$ distributes over unions:
for any $U$ and $V$, $C^e_n(U\cup V)=C^e_n(U)\cup C^e_n(V)$.

\begin{defn}\label{In}
Suppose $X$ is a CW complex with its cells enumerated as $e_i$ 
for $i$ in some index set $I$, 
and for each $i$ in $I$ let $d(i)$ be the dimension of $e_i$.
Then for each $n\in\N$ we let
\[
I^n=\{i\in I\st d(i)\leq n\}.
\]
\end{defn}
Thus, for finite $n$ 
the $n$-skeleton $X^n$ is the union over $i$ in $I^n$ of the cells $e_i$.
Using these notions, we may define an open neighbourhood of a point from a
function to the naturals.  

\begin{defn}
Let $X$ be a CW complex with its cells enumerated as $e_i$ for $i$ in some index
set $I$, and for each $i$ let $d(i)$ be the dimension of $e_i$.
Let $x$ be a point of $X$, lying in cell $e_{i_0}$.
Then for any function $f\from I\to\N$ we define the open neighbourhood
$\UXf{X}{f}{x}$, or simply $\Uf{f}{x}$ when $X$ is clear,
of $x$ in $X$ recursively in dimension as follows.
\begin{itemize}
\item For all $i$ in $I^{d(i_0)}$ other than $i_0$,
we take $\UXf{X}{f}{x}\cap e_{i}=\emptyset$.
\item For $i=i_0$, we take $\UXf{X}{f}{x}\cap e_{i}=B_{f(i)}(x)$.  
\item If $\UXf{X}{f}{x}\cap X^m$ has been defined for some $m\geq d(i_0)$, and 
$i\in I$ is such that $d(i)=m+1$, we set 
\[
\UXf{X}{f}{x}\cap\bar{e}_i=C^{e_i}_{f(i)}(\UXf{X}{f}{x}\cap X^m).
\]
\end{itemize}
\end{defn}

Clearly every such set $\UXf{X}{f}{x}$ is open in $X$.
Note also that if $A$ is a subcomplex of $X$ and $J\subseteq I$ is the set of
indices of cells in $A$, $J=\{i\in I\st e_i\subseteq A\}$, then 
$\UXf{A}{f\restr\,J}{x}=\UXf{X}{f}{x}\cap A$.
We thus use the notation $\Uf{f}{x}$
omitting the superscript without fear of confusion,
with the domain of $f$ dictating the CW complex in which $\Uf{f}{x}$ is taken.

For functions $f\from I\to\N$
we shall write $f\downarrow n$ as a shorthand for the restriction
$f\restr I^n$; thus,
$\Ufn{f}{n}{x}=\Uf{f}{x}\cap X^n$.
In the arguments below we shall even 
use this notation when $f$ has not yet been
defined on $I\smallsetminus I^n$.
Also, as per the set-theoretic convention discussed above,
$f\restr i$ denotes the restriction of $f$ to natural numbers less than $i$, 
$f\restr i=f\restr\{0,\ldots,i-1\}$.

Since each $\Uf{f}{x}$ for $f\from I\to\N$ is open, it will suffice for our proof of
Theorem~\ref{prodCWcharthm} to produce sets of this form.  
In some sense this is also necessary:

\begin{lem}
For any CW complex $X$ with cells $e_i$, $i\in I$, and for any $x$ in $X$,
the sets $\Uf{f}{x}$ as $f$ varies over functions from $I$ to $\N$ 
form an open neighbourhood base at $x$.
\end{lem}
\begin{proof} Given an open neighbourhood $V$
of $x$, we construct recursively on dimension a function
$f\from I\to\N$ such that $\overline{\Uf{f}{x}\cap X^n}\subset V\cap X^n$ 
for every $n\in\N$. 
If $x$ is in cell $e_{i_0}$ of dimension $d(i_0)$, 
then as the base case we may choose
$f(i_0)$ large enough that $B_{f(i_0)}(x)$ has closure contained in $V$,
since $V\cap e_{i_0}$ is open in $e_{i_0}$, and set
$f(i)=0$ for every other $i$ in $I^{d(i_0)}$.
For the inductive step,
suppose we have defined $f$ on $I^n$
in such a way that $\overline{\Ufn{f}{n}{x}}\subset V\cap X^n$, 
and suppose $e_\ell$ is an $(n+1)$-cell of $X$
with characteristic map $\phi_\ell$.  
Then $\phi_\ell^{-1}(\overline{\Ufn{f}{n}{x}})$ 
is a compact subset of $\phi_\ell^{-1}(V)\cap S^n$, 
and thus we may choose $f(\ell)$
sufficiently large that $C^{e_\ell}_{f(\ell)}(\Ufn{f}{n}{x})$ also has closure 
contained in $\phi_\ell^{-1}(V)$.
\end{proof}

We shall repeatedly require the following lemma allowing us to extend open
sets on finite subcomplexes.

\begin{lem}\label{addacell}
Suppose $W$ and $Z$ are CW complexes, $W'$ is a finite subcomplex of $W$,
$Z'$ is a finite subcomplex of $Z$, $U$ is a subset of $W'$ that is open
in $W'$, $V$ is a subset of $Z'$ that is open in $Z'$, 
and $H$ is a sequentially closed subset of $W\times Z$
such that the closure of $U\times V$ is disjoint from $H$.  Let $e$ be a cell
of $Z$ whose boundary is contained in $Z'$.  Then there is a $p\in\N$ such that
$U\times (V\cup C^e_p(V))$ has closure disjoint from $H$.
\end{lem}
The point is that $V\cup C^e_p(V)$ is open in $Z'\cup e$, and we can build up
open sets in the full CW complex $Z$ in this way.  Note also that apart from 
which CW complex $e$ belongs to, 
Lemma~\ref{addacell} is symmetric in $W$ and $Z$, so we
will be able to use it to build up open sets of both $X$ and $Y$ in the proof
of the main theorem. 

\begin{proof}
Denote the subcomplex $Z'\cup e$ of $Z$ by $Z'e$.
The product $W'\times Z'e$ is a compact CW complex, and 
in particular normal and
sequential.  Thus $H\cap (W'\times Z'e)$ is a closed subset of 
$W'\times Z'e$ disjoint from $\overline{U\times V}$, and so we may take
disjoint open sets $\calO_{U\times V}$ and $\calO_H$ in $W'\times Z'e$ such that
$\overline{U\times V}\subseteq\calO_{U\times V}$ and 
$H\cap (W'\times Z'e)\subseteq\calO_H$.
Now, for every point $(u,v)$ of $\overline{U\times V}$,
there is an open base set $R\times S$ of the product topology on 
$W'\times Z'e$ that contains $(u,v)$ and is contained in $\calO_{U\times V}$.
By shrinking $S$ if necessary, we may assume $S$ is of the form 
$T\cup C^e_n(T)$ for some open subset $T$ of $Z'$ and some $n\in\N$ 
(recall that this also makes sense if $T\cap\bar{e}$ is empty, in which case 
$n$ is arbitrary).
Now, by compactness of $\overline{U\times V}$,
finitely many such base sets $R\times S$ suffice to cover 
$\overline{U\times V}$, and we may choose $p\in\N$
to be strictly greater than all of the corresponding values $n$.
Then $U\times(V\cup C^e_p(V))$ has closure contained in $\calO_{U\times V}$, 
and hence
disjoint from $H$, as required.
\end{proof}

We are now ready to prove our main result.  
By Proposition~\ref{lltkacomponents}, the formulation given here is equivalent
to Theorem~\ref{prodCWcharthm}.

\begin{thm}\label{prodCWcharthmloc}
Let $X$ and $Y$ be CW complexes.  Then $X\times Y$ is a CW complex if and only if
one of the following holds:
\begin{enumerate}
\item $X$ or $Y$ is locally finite.
\item One of $X$ and $Y$ is locally countable, and the other is locally 
less than $\frb$.
\end{enumerate}
\end{thm}

\begin{proof}
As discussed at the start of this section, it suffices to show that if
$X$ has countably many cells and $Y$ has fewer than $\frb$ many cells, then
$X\times Y$ is a CW complex.
So suppose $X$ is a CW complex with countably many cells and 
$Y$ is a CW complex with fewer than $\frb$ many cells.
We shall show that the product topology on $X\times Y$ is sequential, and so indeed
makes $X\times Y$ a CW complex.  
To this end, let
$H$ be an arbitrary sequentially closed subset of $X\times Y$, and take
$(x_0,y_0)\in X\times Y\smallsetminus H$; 
we wish to construct an open neighbourhood
of $(x_0,y_0)$ disjoint from $H$.

Enumerate the cells of $X$ as $e_{X,i}$ for $i$ in $\N$, in such a way that for
each $i$, the boundary of $e_{X,i}$ is contained in $\bigcup_{j<i}e_{X,i}$ ---
this is possible by closure-finiteness.
We define the finite subcomplex $X_i$ of $X$ to be
\[
X_i=\bigcup_{j\leq i}e_{X,i}.
\]
Enumerate 
the cells of $Y$ as $e_{Y,\al}$ for $\al$ in some index set $J$ with cardinality
$\mu<\frb$
(we leave $J$ abstract rather than declaring $J=\mu$ so that the
notation $J^n$ of Definition~\ref{In} remains clear).
Recall our notation $\minsub{Y}{\al}$ from Definition~\ref{Xminal} for the
minimal 
subcomplex of $Y$ containing $e_{Y,\al}$.
Let $m(i)$ be the dimension of cell $e_{X,i}$, and let $n(\al)$ be the dimension
of cell $e_{Y,\al}$.
Let $e_{X,i_0}$ be the unique open cell of $X$ containing $x_0$,
and $e_{Y,\al_0}$ the unique open cell of $Y$ containing $y_0$.
We shall construct functions $f\from \N\to\N$ and $g\from J\to\N$
such that $\Uf{f}{x_0}\times \Uf{g}{y_0}$ is disjoint from $H$.
As ever, the construction is by recursion, but we shall recurse over dimension
on the $Y$ side and over $i$ on the $X$ side, whilst also keeping track of a 
lower bound function for the $X$ side.
Specifically, we shall construct for each $i$ in $\N$ functions
$f_i\from\N\to\N$ and $g_i\from J^{n(\al_0)+i}\to\N$ such that 
\renewcommand{\theenumi}{\alph{enumi}}
\begin{enumerate}
\item\label{cdfH} $\Uf{f_i}{x_0}\times\Uf{g_i}{y_0}$ 
has closure disjoint from $H$,
\item\label{extends} for all $j>i$, $g_j\downarrow{n(\al_0)+i} = g_i$,
$f_j\restr i=f_i\restr i$, and for all $n\geq i$, $f_j(n)\geq f_i(n)$.
\end{enumerate}
\renewcommand{\theenumi}{\arabic{enumi}}
With such functions in hand we may define $f$ and $g$ by
$f(i)=f_{i+1}(i)$ and $g(\al)=g_{n(\al)-n(\al_0)}(\al)$.
Then 
\begin{align*}
\Uf{f}{x_0}\times\Uf{g}{y_0}
&=\bigcup_{i\in\N}\Ufrestrn{f}{i}{x_0}\times\Ufn{g}{n(\al_0)+i}{y_0}\\
&=\bigcup_{i\in\N}\Ufrestrn{f_i}{i}{x_0}\times\Uf{g_i}{y_0},
\end{align*}
each term of which will be disjoint from $H$ by construction.

\begin{aside}
The reader familiar with Hechler forcing (the simplest 
forcing partial order used for constructing models of set theory with 
a specified value of $\frb$)
will note that the requirements on the functions $f_i$ in (\ref{extends})
are that they form a descending sequence of Hechler conditions,
where the stem of $f_i$ is $f_i\restr i$.
\end{aside}

For the base case of the construction, consider 
$X\times\minsub{Y}{\al_0}$.
Since $\minsub{Y}{\al_0}$ is a finite CW complex, 
$X\times\minsub{Y}{\al_0}$
is a CW complex, $(X\times \minsub{Y}{\al_0})\cap H$ is closed, and
we may choose a function $f_0\from\N\to\N$ and a natural number number
$g_0(\al_0)$ such that 
$\Uf{f_0}{x_0}\times B_{g_0(\al_0)}(y_0)$ has closure disjoint from $H$.
For $\al\neq\al_0$ in $J^{n(\al_0)}$, set $g_0(\al)=0$, so we have
$g_0$ defined on all of $J^{n(\al_0)}$; 
since $\Uf{g_0}{y_0}=B_{g_0(\al_0)}(y_0)$, we have that
$\Uf{f_0}{x_0}\times \Uf{g_0}{y_0}$ has closure disjoint from $H$.

\begin{lem}\label{fnforal}
Let $Y'$ be a finite subcomplex of $Y$ containing $y_0$, 
let $F$ be a function from $\N$ to $\N$ and 
$s$ a function from the indices of $Y'$ to $\N$
such that $\Uf{F}{x_0}\times \Uf{s}{y_0}\subseteq X\times Y'$ 
has closure disjoint from 
$H$.
Let $i$ be a natural number and let $Y''$ be a subcomplex of $Y$ that is a 
one cell extension of $Y'$, $Y''=Y'\cup e_\al$.
Then there is a function $f\from\N\to\N$ 
such that
\begin{enumerate}
\item \label{fdomF}
$f(n)\geq F(n)$ for all $n$ in $\N$, and $f(n)=F(N)$ for all $n<i$,
\item \label{allgreatergood}
for every $f'\from\N\to\N$ such that $f'\geq^* f$ and $f'\geq F$,
there is a $q\in\N$ such that
$\Uf{f'}{x_0}\times \Uf{s\cup\{(\al,q)\}}{y_0}$ has closure disjoint from $H$.
\end{enumerate}
\end{lem}

\begin{proof}

The construction of $f$ is by recursion on $n\geq i$, with 
repeated applications of Lemma~\ref{addacell}.
As the base case, set $f\restr i= F\restr i$. 

Suppose we have constructed $f\restr n$.  For every sequence 
$r\from n\to\N$ such that $F(m)\leq r(m)\leq f(m)$ for all $m<n$,
let
$q(r)$ be the least $q\in\N$ such that
$\Uf{r}{x_0}\times \Uf{s\cup\{(\al,q)\}}{y_0}$ has closure disjoint from $H$;
such a $q$ must exist by assumption on $F$ and $s$ and Lemma~\ref{addacell}.
Then let $p(r)$ be the least $p\in\N$ such that
$\Uf{r\cup\{(n,p)\}}{x_0}\times\Uf{s\cup\{(\al,q(r))\}}{y_0}$ 
has closure disjoint from $H$, again applying Lemma~\ref{addacell}.
Finally, we define $f(n)$ as
\[
f(n)=\max(\{p(r)\st F\restr n\leq r\leq f\restr n\}\cup\{F(n)\}).
\]
We claim that this recursive construction yields a function $f\from\N\to\N$
as per the statement of the lemma. 

(\ref{fdomF}) is immediate from the construction.
For (\ref{allgreatergood}), suppose $f'\from\N\to\N$ is such that
$f'\geq^*f$ and $f'\geq F$.
Let $n_0\in\N$ be such that for all $n\geq n_0$, $f'(n)\geq f(n)$.
Let $r$ be the $n_0$-tuple defined by $r(m)=\min(f(m), f'(m))$.
Note that $r\restr i=f\restr i=F\restr i$.  The natural number $q(r)$ is
then a $q$ as required by (\ref{allgreatergood}). 
Indeed we shall show by induction that, letting $f''$ be the function
\[
f''(n)=\begin{cases}
r(n)&\text{if }n<n_0\\
f(n)&\text{if }n\geq n_0,
\end{cases}
\]
we obtain that
$\Uf{f''}{x_0}\times \Uf{s\cup\{(\al,q(r))\}}{y_0}$ has closure disjoint from $H$.
The result will then follow, as $f'\geq f''$ and hence
$\Uf{f'}{x_0}\subseteq\Uf{f''}{x_0}$.

For the base case, 
$\Uf{r}{x_0}\times \Uf{s\cup\{(\al,q(r))\}}{y_0}$ has closure disjoint from $H$
by definition of $q(r)$. 
For $n\geq n_0$, suppose we have shown that
$\Ufrestrn{f''}{n}{x_0}\times\Uf{s\cup\{(\al,q(r))\}}{y_0}$ 
has closure disjoint from $H$.  Then by minimality $q(f''\restr n)\leq q(r)$.  
Also $f''(n)=f(n)\geq p(f''\restr n)$;
so 
$\Ufrestrn{f''}{n+1}{x_0}\times \Uf{s\cup\{(\al,q(f''\restr\,n))\}}{y_0}$ has closure disjoint from $H$, whence the possibly smaller set
$\Ufrestrn{f''}{n+1}{x_0}\times \Uf{s\cup\{(\al,q(r))\}}{y_0}$ has closure disjoint from $H$, as required for the inductive step.
We therefore have that for every $n$,
$\Ufrestrn{f''}{n}{x_0}\times\Uf{s\cup\{(\al,q(r))\}}{y_0}$ has closure disjoint from
$H$. Since
\[
\bigcup_{n\in\N}\overline{\Ufrestrn{f''}{n}{x_0}\times\Uf{s\cup\{(\al,q(r))\}}{y_0}}
\]
is closed in every cell of $X\times Y''$, it is closed in $X\times Y''$,
and so
$\Uf{f''}{x_0}\times\Uf{s\cup\{(\al,q(r))\}}{y_0}$ has closure disjoint from
$H$, as required.
\end{proof}

Returning to the construction of the functions $f_i$ and $g_i$ for $i$ in $\N$,
suppose that for all $j\leq k$
we have constructed the functions $f_j\from\N\to\N$ and 
$g_j\from J^{n(\al_0)+j}\to\N$ satisfying the requirements
(\ref{cdfH}) and (\ref{extends}) above. 
Call $\al\in J$ \emph{relevant} if $n(\al)=n(\al_0+k+1)$; these are the indices 
we need to extend the definition of $g$ to for the inductive step.

For relevant $\al$, 
let $Y'_\al$ be $(\minsub{Y}{\al_0}\cup\minsub{Y}{\al})\smallsetminus e_\al$
and let $J_\al$ be the set of indices of cells in $Y'_\al$.
Apply Lemma~\ref{fnforal} with 
$f_k$ as $F$,
$\minsub{Y}{\al_0}\cup\minsub{Y}{\al}$ as $Y''$,
$Y'_\al$ 
as $Y'$,
$g_k\restr J_\al$ as $s$,
and $k+1$ as $i$.
The requirement of the lemma that 
$\Uf{f_k}{x_0}\times\Ufrestrn{g_k}{J_\al}{y_0}$ 
have closure disjoint from $H$ holds by the inductive hypothesis.
We thus get for each relevant $\al$ a function $f_{k+1,\al}$
satisfying (\ref{fdomF}) and (\ref{allgreatergood}) of Lemma~\ref{fnforal}.
Since there are fewer than $\frb$ many members of $J$, there is a single
function $f_{k+1}\from\N\to\N$ that eventually dominates $f_{k+1,\al}$
for every relevant $\al$. 
Taking $f_{k+1}$ as $f'$ in 
(\ref{allgreatergood}) of Lemma~\ref{fnforal}, we have that for each 
relevant $\al$ 
there is $q_\al\in\N$ such that the open subset
$\Uf{f_{k+1}}{x_0}\times 
\Uf{(g_k\restr J_{\al})\cup\{(\al,q_\al)\}}{y_0}$
of $X\times(\minsub{Y}{\al_0}\cup\minsub{Y}{\al})$
has closure disjoint from $H$.
Products commute with closures in the product topology, so
\begin{multline*}
\overline{\Uf{f_{k+1}}{x_0}\times 
\Uf{(g_k\restr J_{\al})\cup\{(\al,q_\al)\}}{y_0}}\\
=
\overline{\Uf{f_{k+1}}{x_0}}\times 
\overline{\Uf{(g_k\restr J_{\al})\cup\{(\al,q_\al)\}}{y_0}},
\end{multline*}
and since $Y^{n(\al_0)+k+1}$ has the weak topology,
\begin{multline*}
\overline{\bigcup_{\al\text{ relevant}}\Uf{(g_k\restr J_{\al})\cup\{(\al,q_\al)\}}{y_0}}\\
=
\bigcup_{\al\text{ relevant}}\overline{\Uf{(g_k\restr J_{\al})\cup\{(\al,q_\al)\}}{y_0}}.
\end{multline*}
So $\Uf{f_{k+1}}{x_0}\times 
\bigcup_{\al\text{ relevant}}\Uf{(g_k\restr J_{\al})\cup\{(\al,q_\al)\}}{y_0}$
is an open subset of $X\times Y^{n(\al_0)+k+1}$ with closure disjoint from $H$.
Since $f_{k+1}\geq f_k$, $\Uf{f_{k+1}}{x_0}\subseteq\Uf{f_k}{x_0}$,
and so $\Uf{f_{k+1}}{x_0}\times\Uf{g_k}{y_0}$ has closure disjoint from $H$.
Taking $g_{k+1}=g_k\cup\{(\al,q_\al)\st\al\text{ is relevant}\}$ completes the
inductive step.

We thus have a recursive construction of the functions $f_i$ and $g_i$ as
required, which as discussed above allows us to form the functions $f$ and
$g$ defining an open neighbourhood $\Uf{f}{x_0}\times\Uf{g}{y_0}$ of
$(x_0,y_0)$ disjoint from $H$.  Since $H$ was an arbitrary sequentially closed
subset of $X\times Y$ and $(x_0,y_0)$ was an arbitrary point in the complement
of $H$ in $X\times Y$, this shows that $X\times Y$ is sequential, and
thus bears the weak topology.  That is, $X\times Y$ is a CW complex.
\end{proof}

\section*{Acknowledgements}

The author would like to thank Jeremy Rickard for drawing his attention to
this question, and Nicola Gambino, Dugald Macpherson and Severin Mejak
for helpful comments on drafts of this paper.  
The author also gratefully acknowledges
the support for this work of the EPSRC through Early Career Fellowship
EP/K035703/2, \emph{Bringing set theory and algebraic topology together}.

\bibliographystyle{plain}
\bibliography{ABT}

\end{document}